\documentclass{amsart}
\usepackage{graphicx}
\usepackage{subcaption}

\newtheorem{theorem}{Theorem}[section]
\newtheorem{lemma}[theorem]{Lemma}
\newtheorem{corollary}[theorem]{Corollary}

\theoremstyle{definition}

\theoremstyle{remark}

\numberwithin{equation}{section}

\newcommand*\diff{\mathop{}\!\mathrm{d}}
\DeclareMathOperator{\Span}{span}
\DeclareMathOperator{\csch}{csch}
\DeclareMathOperator{\sech}{sech}

\begin{document}

\title{Heat Content Determines Planar Triangles}

\author{Reed Meyerson}
\address{Department of Mathematics, New College of Florida, Sarasota, Florida 34243}
\curraddr{}
\email{reed.meyerson12@ncf.edu}
\thanks{}

\author{Patrick McDonald}
\address{Department of Mathematics, New College of Florida, Sarasota, Florida 34243}
\email{mcdonald@ncf.edu}
\thanks{}

%    General info
\subjclass[2010]{Primary 58J50, 58J35; Secondary 58J65}

\date{June 1, 2016}

\dedicatory{}

\keywords{spectral theory, heat content, exit time moments}

\begin{abstract}
We prove that the heat content determines planar triangles.
\end{abstract}

\maketitle

\section{Introduction}
Let $D$ be a bounded open subset of $\mathbb{R}^2.$   Suppose we heat $D$ to uniform initial temperature of 1, and then, holding the boundary of $D$ at temperature 0, we let the heat dissipate.  We can describe the evolution of temperature via the solution of the heat equation on $D.$  To do so, let $\Delta$ be the Laplacian\footnote{We work with the convention that the Dirichlet Laplacian is positive.} on  $D$ and solve  
\begin{align}
\frac{\partial u}{\partial t} & = \Delta u   \quad\text{on } D \times (0,\infty) \label{HE1.1} \\
\lim\limits_{x\rightarrow x_0} u(x,t)  & = 0 \quad\forall\text{ } x_0 \in \partial D \times (0,\infty) \label{HE1.2} \\
u(x,0) & = 1 \quad\text{on } D\label{HE1.3} 
\end{align}
Using the temperature $u(x,t),$ we can associate to $D$ a measure of the heat in $D$ at time $t,$ the so-called heat content of $D:$ 
\begin{equation}
H_D(t) = \int\limits_{D}u(x,t) \diff x \label{HC1.1}
\end{equation}
We prove:

\begin{theorem}\label{mainresult} Heat content determines planar triangles.
\end{theorem}

Theorem (\ref{mainresult}) should be viewed in the context of related results from spectral geometry.   In particular, in her thesis C. Durso \cite{D} proved 

\begin{theorem} Dirichlet spectrum determines planar triangles.
\end{theorem}

Durso used wave trace methods to establish her theorem.  A recent paper of Grieser and Marrona \cite{GM} establshes the result using heat trace.  We follow the argument of \cite{GM} to establish our result.

The relationship of Dirichlet spectrum to the geometry of the underlying domain is a well-studied topic with an extensive associated literature.  The same is true for the relationship between heat content and geometry, but the associated geometric invariants are distinct.  In particular, heat content is not spectral; that is, it is not determined by the Dirichlet spectrum of the domain $D.$  In fact, if we denote by $\lambda_n$ the $n$th Dirichlet eigenvalue enumerated in increasing order with multiplicity, and by $\phi_n$ a collection of associated orthonormal eigenfunctions, then 
\begin{equation*}
H_D(t) = \sum A_n^2 e^{-\lambda_n t}
\end{equation*}
 where the coefficient $A_n = \int_D \phi_n(x) \diff x$ contributes off-diagonal information (for more on the relationship between Dirichlet spectrum and heat content, see \cite{G}, \cite{MM}).  

Heat content is closely related to Brownian motion and its associated exit times from a given domain.  In more detail, suppose $X_t$ is Brownian motion in the plane, with ${\mathbb P}^x$ the associated collection of probability measures charging Brownian paths beginning at $x.$  Given a domain $D\subset {\mathbb R}^2,$ let $\tau$ be the first exit time from $D:$
\begin{equation*}
\tau = \inf\{t\geq 0: X_t \notin D\}.
\end{equation*}
Let ${\mathbb E}^x$ be expectation with respect to the probability ${\mathbb P}^x$ and consider the moments of the exit time, $ {\mathbb E}^x[\tau^k],$ where $k$  varies over the natural numbers.  Integrating over starting points in the domain $D$ results in a sequence of positive real numbers, the $L^1$-norms $ \|{\mathbb E}^x[\tau^k]\|_1,$ of the exit time moments of Brownian motion.  It is easy to see that each element in the sequence is an invariant of the isometry class of $D.$  The extent to which this $L^1$-moment spectrum determines the geometry of $D$ is an active area of research, with roots in the nineteenth century (the first moment is also known as the torsional rigidity, a well-studied construct in the theory of elastica, and the objective function in the St. Venant Torsion Problem).    For domains which are sufficiently regular, it is known that the $L^1$-moment spectrum determines heat content \cite{MM}.  From this we have an immediate corollary:

\begin{corollary} The $L^1$-moment spectrum determines planar triangles.
\end{corollary}

In the remainder of the paper we prove our main result.

\section{Proof of Theorem \ref{mainresult}}

As in the introduction, let $D$ be a bounded, open subset of $\mathbb{R}^2$ and let $u(x,t)$ denote the solution of the heat equation with uniform initial temperature distribution 1 and Dirichlet boundary conditions (i.e., $u(x,t)$ solves (\ref{HE1.1})-(\ref{HE1.3})).  Let $H_D(t)$ be the heat content of $D:$
\begin{equation*}
H_D(t) = \int\limits_{D}u(x,t) \diff x.
\end{equation*}

As mentioned in the introduction, the relationship of heat content to the geometry of the underlying domain is a well-studied topic with an extensive associated literature.  For our purposes, it is known that when the boundary of $D$ is sufficiently regular, there is a small time asymptotic expansion of $H_D(t).$  More precisely, when $D$ is smoothly bounded domain with compact closure in a Riemannian manifold, it is a theorem of Van den Berg and Gilkey \cite{BG} that $H_D(t)$ has an expansion of the form 
\begin{equation}\label{asymptotics2.1}
H_D(t) \simeq  \sum_{n=0}^\infty a_n t^{\frac{n}{2}}
\end{equation}
where the coefficients $a_n$ are local invariants of the metric.  A number of the coefficients have been computed; for example, it is known that $a_0= |D|$ and $a_1 = -\frac{2}{\sqrt{\pi}} |\partial D|$ where $|D|$ denotes the Riemannian volume of $D$ and $|\partial D|$ denotes the volume of the boundary of $D$ with respect to the induced surface measure.  (For a more complete discussion of invariants appearing in the expansion (\ref{asymptotics2.1}), and the relationship between heat content, heat trace and Dirichlet spectrum see \cite{G} and references therein). 
 
When $D$ is a planar polygon, it is a theorem of Van den Berg and Srisatkunarajah  \cite{BS}  that $H_D$ has small time asymptotic expansion given by
\begin{equation}\label{asymptotics2.2}
H_D(t) = |D| - \frac{2|\partial D|}{\sqrt{\pi}} t^{\frac{1}{2}}
+ 4t\sum\limits_{i=1}^n\varphi(\theta_i) + O(e^{-q/t})
\end{equation}
where $|D|$ is the area of $D,$ $|\partial D|$ is the perimeter of the boundary of $D,$ $\{\theta_i\}$ are the interior angles associated to $D,$ and 
\begin{equation}\label{thirdterm}
\varphi(\theta) = \int\limits_0^\infty\frac{\sinh[(\pi-\theta)\xi]}{\sinh(\pi\xi)\cosh(\theta\xi)} \diff \xi.
\end{equation}
\par
In the same paper, Van den Berg and Srisatkunarajah also establish a small time asymptotic expansion for the heat trace.  More precisely, if $h_D(t)$ is the trace of the heat kernel, then, with notation as above, 
\begin{equation}\label{asymptotics2.3}
h_D(t) = \frac{|D|}{4\pi} t^{-1} - \frac{|\partial D|}{8\sqrt{\pi}}t^{-\frac{1}{2}}  
+ \frac{1}{24} \sum\limits_{i=1}^n \left(\frac{\pi}{\theta_i} - \frac{\theta_i}{\pi}\right) + O(e^{-r/t})
\end{equation}
\par

In \cite{GM}, Grieser and Maronna use the heat trace asymptotics to show that a triangle's Dirichlet spectrum is unique amongst triangles. To prove their theorem, they show that triangles are determined up to isometry by a triple given by area, perimeter, and the sum of the reciprocals of the interior angles (up to a constant, the third term appearing in the asymptotics of the heat trace for the domain).   We employ a similar analysis to prove Theorem \ref{mainresult}.  More precisely,  let  $\Phi:(0,\pi)\times(0,\pi)\times(0,\pi)\rightarrow\mathbb{R}$ be defined by
\begin{equation}\label{G}
\Phi(x,y,z) = \varphi(x)+\varphi(y)+\varphi(z)
\end{equation}
with $\varphi$ as in (\ref{thirdterm}), the summands in the third term in the asymptotics of the heat content for the domain.  We will prove that triangles are determined up to isometry by the first three coefficients of the heat content for the domain; i.e. the triple given by area, perimeter, and value of $\Phi$ as a function of the interior angles. This is equivalent to the following:

\begin{theorem}\label{thm:main}
Let $T$ be the space of equivalence classes of isometric triangular domains in the plane. Define $H:T\rightarrow C^\infty(0,\infty)$ by
\begin{equation*}
H(D) = H_D(t), \quad\text{ for } D\in T
\end{equation*}
where $H_D(t)$ is the heat content associated to $D.$  Then $H$ is injective.
\end{theorem}

From equation (\ref{asymptotics2.2}), heat content determines the area of our domain.  Thus, we can restrict our attention to the space of triangles up to scaling, which we denote $T_s$. 

We will use the following notation:  Given a function $F:\mathbb{R}^3_+\rightarrow\mathbb{R}$ and a real number $r$,  let $L_r(F)$ be the level set of $F$ with value $r:$ 
\begin{equation*}
L_r(F) = \{(x,y,z)\in\mathbb{R}^3_+:F(x,y,z)=r\}
\end{equation*}

Where $\mathbb{R}^3_+$ denotes the positive octant of $\mathbb{R}^3$. Define a function $\Theta:\mathbb{R}^3_+\rightarrow\mathbb{R}$ by
\begin{equation}\label{theta}
\Theta(\alpha,\beta,\gamma) = \alpha + \beta + \gamma.
\end{equation}

Let $\mathbb{T} = L_\pi(\Theta)$. Then $T_s$ may be identified with quotient $\mathbb{T}/\sim$ where the equivalence relation $\sim$ identifies permutations of a given triple.  Because each ordering of the angles corresponds to one of the six sections in figure~(\ref{fig:head_on}), there is a bijective correspondence between $T_s$ and any one of the six sections of $\mathbb{T}$. We will call a point in $\mathbb{T}$ an isosceles point if it represents an isosceles triangle. This occurs if and only if $p$ lies on one of the three lines in figure~(\ref{fig:head_on}). The point at the intersection of these lines represents an equilateral triangle, thus we will call it the equilateral point.
\par
\begin{figure}[h]
	\centering
	\begin{subfigure}[b]{0.5\textwidth}
    	\includegraphics[width=\textwidth]{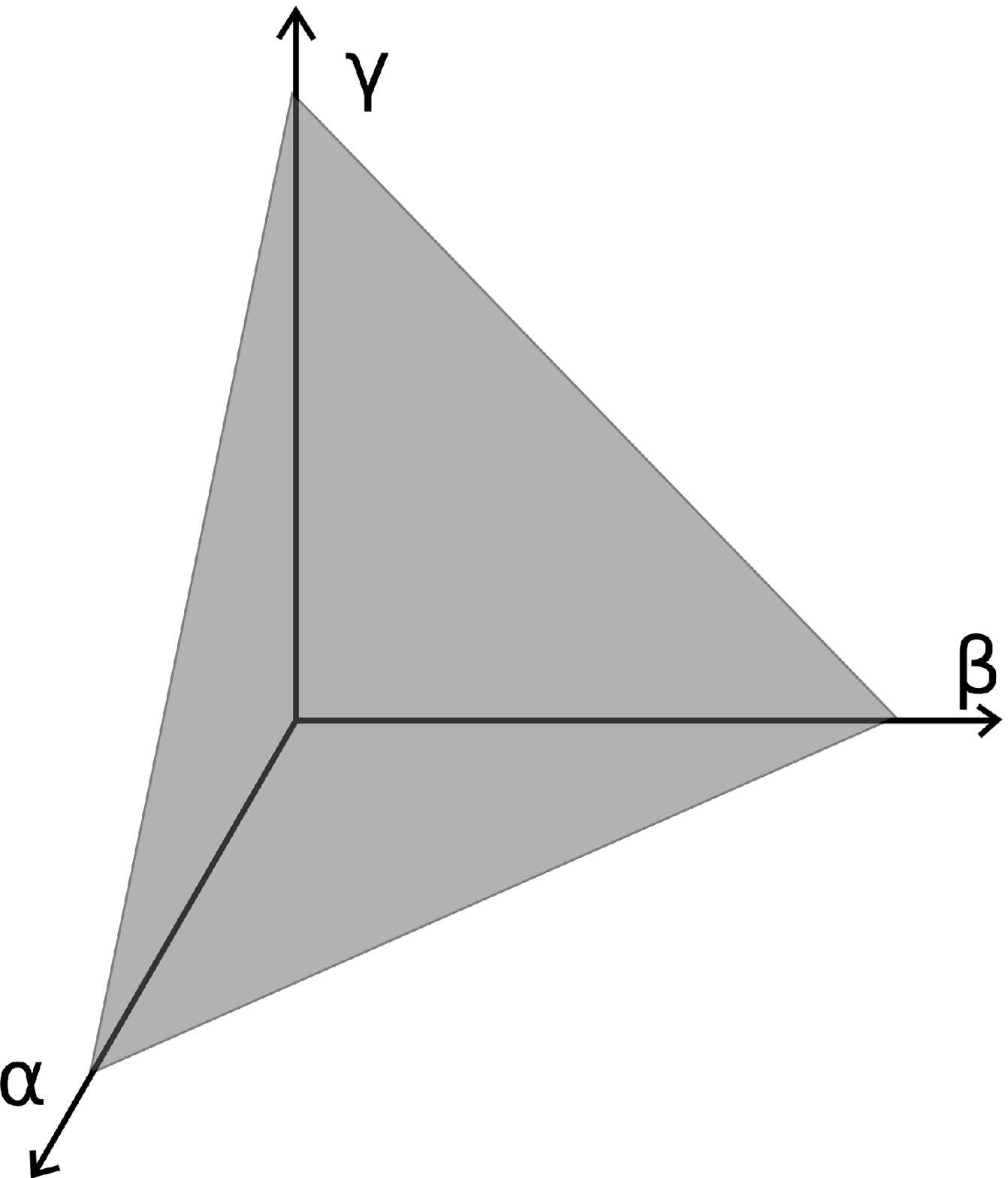}
  		\caption{$L_\pi(\Theta)$}
  		\label{fig:in_r3}
	\end{subfigure}
	~
	\begin{subfigure}[b]{0.5\textwidth}
		\includegraphics[width=\textwidth]{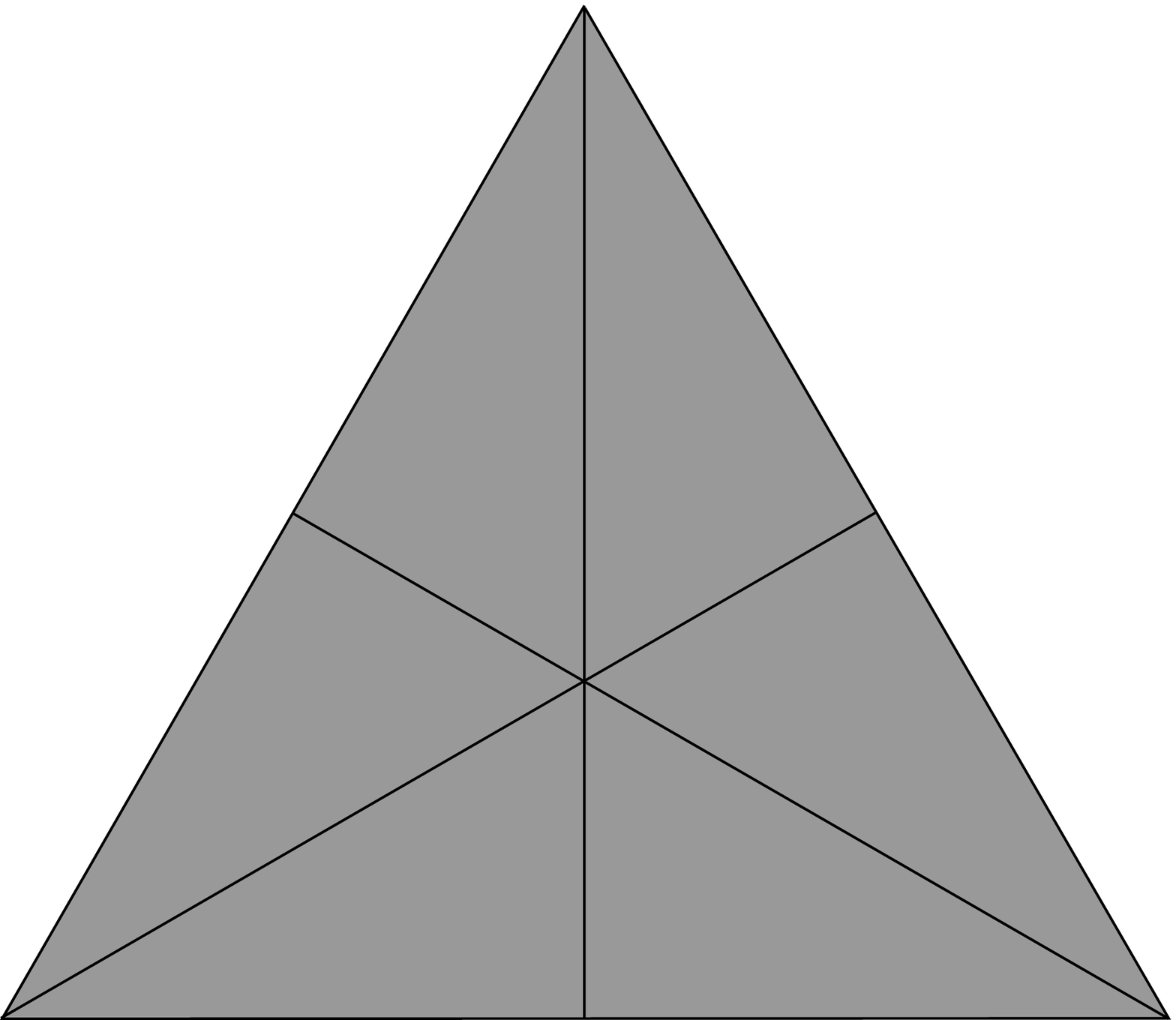}
		\caption{The six sections of $L_\pi(\Theta)$ each correspond to a different ordering of the angles}
		\label{fig:head_on}
	\end{subfigure}
	\label{fig:planes}
	\caption{Two views of $L_\pi(\Theta)$}
\end{figure}

Write 
\begin{equation}\label{cot}
\psi(x) = \cot\left(\frac{x}{2}\right)
\end{equation}
and let $\Psi:\mathbb{R}_3^+\rightarrow\mathbb{R}$ be defined by
\begin{equation}\label{psi}
\Psi(\alpha,\beta,\gamma) = \psi(\alpha) + \psi(\beta) + \psi(\gamma).
\end{equation}
It can be shown using elementary geometry \cite{GM}, that for any triangle with interior angles $(\alpha,\beta,\gamma),$
\begin{equation*}
\Psi(\alpha,\beta,\gamma) = \frac{|\partial D|^2}{4|D|}.
\end{equation*}
Observe that $\Psi(p)\to \infty$ as $p$ approaches the boundary of $\mathbb{T}$. Thus, the sublevel sets of $\Psi$, given by $S_c(\Psi)=\{p\in\mathbb{T}:\Psi(p)\le c\}$, stay away from the boundary of $\mathbb{T}$. In addition, a straightforward computation of the Hessian of $\Psi$ indicates that $\Psi$ is strictly convex on the portion of the positive octant bounded by $\mathbb{T}.$ Thus $S_c$ is a convex set.
\par 
It follows from the definition of $\Psi$ that its level sets and sublevel sets will inherit the perumtation symmetry of $\mathbb{T}/\sim$. Using this symmetry and the convexity of $S_c$, it can be argued that $L_c(\Psi)\cap\mathbb{T}$ contains the equilateral point if and only if $L_c(\Psi)\cap\mathbb{T}$ is a singleton. Thus, if $F:T_s\rightarrow\mathbb{R}$ can be written as the sum of a strictly convex function of angles, then the equilateral triangle is determined by its area and the value of $F$.
\par 
Before proving Theorem~\ref{mainresult}, we recall the method of Lagrange multipliers.

\begin{lemma}\label{lagrangemultiplier}
Let $f,g,h:\mathbb{R}^3\rightarrow\mathbb{R}$ be $C^1$. Let $x_0\in\mathbb{R}$. Suppose $\nabla g(x_0)$ and $\nabla h(x_0)$ are linearly independent. If $f(x_0)$ is a local solution to the problem
\begin{itemize}
    \item[] Maximize $f(x)$
    \item[] Subject to $g(x)=g(x_0)$, $h(x)=h(x_0)$
\end{itemize}
Then $\nabla f(x_0)\in\Span[\nabla g(x_0),\nabla h(x_0)]$.
\end{lemma}

Theorem~\ref{mainresult} is a corollary of the following lemmas which we will prove in the sequel:

\begin{lemma}\label{lin_ind}
    Let $f,g:(0,\pi)\rightarrow\mathbb{R}$ be $C^1$, monotone decreasing and convex. Suppose there exists a real $c>0$ such that $f'-c g'$ is increasing and convex. Suppose $F(\alpha,\beta,\gamma) = f(\alpha)+f(\beta)+f(\gamma),$  $G(\alpha,\beta,\gamma) = g(\alpha)+g(\beta)+g(\gamma)$ and $\Theta$ is as in (\ref{theta}).  Then $\nabla F, \nabla G, \nabla \Theta$ are linearly independent for all non-isosceles points of $\mathbb{T}.$
\end{lemma}

\begin{lemma}\label{log2_lemma}
    Let $\varphi,\psi$ be as defined in (\ref{thirdterm}) and (\ref{cot}), respectively. Then $\varphi,\psi$ are decreasing and convex on $(0,\pi)$, and $\varphi'-\frac{\log 2}{2}\psi'$ is increasing and convex on $(0,\pi)$. 
\end{lemma}

\begin{proof}[Proof of Theorem~\ref{mainresult}]
    Let $t\in T_s$. Choose a representative $p\in\mathbb{T}$ of $t$. By the discussion preceding Lemma~\ref{lagrangemultiplier}, we may assume $p$ is not the equilateral point. Then $L_{\Psi(p)}(\Psi)\cap\mathbb{T}$ is a closed curve around the equilateral point. Let $L$ be a segment of $L_{\Psi(p)}(\Psi)\cap\mathbb{T}$ such that $L$ contains $p$, $L$ is contained in one of the six sections of $\mathbb{T}$, and $L$ begins and ends at isosceles points. It follows that $L$ contains a representative of every triangle which agrees with $t$ on the value of $\Psi$. Thus, if $\Phi$ is monotone on $L$, then $\Phi$ differentiates between triangles with a fixed area and fixed value of $\Psi$. 
    \par 
    Suppose $\Phi$ is not monotone on $L$. Then $\Phi\bigg|_L$ reaches a local extremum at a non-isosceles point. A contradiction follows from Lemma~\ref{lagrangemultiplier}.
\end{proof}
\begin{proof}[Proof of lemma~\ref{lin_ind}]
    Suppose $\nabla F(x_0)+A\nabla G(x_0)-B\Theta(x_0)=0$. Then there are at least three solutions to the equation $f'(x_0)+Ag'(x_0)=B$. Suppose $\tilde f = f'-cg'$ is increasing and convex. Consider the equation $\tilde f(x) + c_1g(x) = c_2$ for fixed $c_1$ and $c_2$. If $c_1$ is positive, the left hand side is convex and there may be at most two solutions. If $c_1$ is negative, the left hand side is increasing and there may be at most one solution.
\end{proof}
\begin{proof}[Proof of lemma~\ref{log2_lemma}]
Write
\begin{align*}
    \tilde \psi(x) &= \psi(\pi x)-\frac{2}{\pi x}\\
    \tilde \varphi(x) &= \varphi(\pi x)-\frac{\log 2}{\pi x}.
\end{align*}
Then
\begin{equation*}
    \varphi(\pi x)-\frac{\log 2}{2}\psi(\pi x) = \tilde \varphi(x)-\frac{\log 2}{2}\tilde \psi(x).
\end{equation*}

Thus, it will be sufficient to show that $\tilde \varphi'-\frac{\log 2}{2}\tilde \psi'$ is increasing and concave up on $(0,1)$. We change variables in the definition of $\varphi$ to obtain
\begin{equation*}
    \varphi(\pi x) = \frac{1}{\pi}\int\limits_0^\infty\frac{\sinh(y-xy)}{\sinh(y)\cosh(xy)}dy.
\end{equation*}
Using the angle addition formula for $\sinh(x),$ we can write
\begin{equation*}
    \varphi(\pi x) = \frac{1}{\pi}\int\limits_0^\infty\coth(y)[\tanh(y)-\tanh(xy)]dy.\\
\end{equation*}
It may be verified with elementary calculus that $\psi$ is decreasing and convex. Additionally, by observing the signs of $\frac{\partial}{\partial x}\tanh(xy)$ and $\frac{\partial^2}{\partial x^2}\tanh(xy)$, it follows that $\varphi$ is decreasing and convex.
\par
We continue with our manipulation of $\varphi$ in an attempt to acquire a manageable form of $\tilde \varphi$. Integration by parts yields
\begin{align}
    \varphi(\pi x) =& \lim_{\stackrel{\epsilon \to 0}{\eta \to \infty}}\frac{1}{\pi}\coth(y)\left[\log(\cosh(y))-\frac{\log(\cosh(xy))}{x}\right]\bigg|_{y=\epsilon}^{y=\eta}\label{lims}\\
    +&\frac{1}{\pi}\int\limits_0^\infty \left(\log(\cosh(y))-\frac{\log(\cosh(xy))}{x}\right)\csch^2(y)dy. \nonumber 
\end{align}
Applying L'Hopital's rule we can evaluate (\ref{lims}).
\begin{equation*}
    \lim_{\stackrel{\epsilon \to 0}{\eta \to \infty}} \frac{1}{\pi}\coth(y)\left[\log(\cosh(y))-\frac{\log(\cosh(xy))}{x}\right]\bigg|_{y=\epsilon}^{y=\eta} =\frac{\log 2}{\pi x}-\frac{\log 2}{\pi}.
\end{equation*}
Noting that
\begin{equation*}
    \int\limits_0^\infty \log(\cosh(y))\csch^2(y)dy = \log 2,
\end{equation*}
we obtain
\begin{equation}\label{varphi_rep}
    \tilde \varphi(x) = -\frac{1}{\pi}\int\limits_0^\infty\frac{\log(\cosh(xy))}{x}\csch^2(y) dy.
\end{equation}
Using this representation of $\tilde \varphi$, we compute relevant derivatives.  Setting
\begin{equation*}
    I(x) = \frac{\log(\cosh x)}{x}
\end{equation*}
we have 
\begin{align}
    \tilde \varphi^{(k)}(x) &= -\frac{1}{\pi}\int\limits_0^\infty y^{k+1}I^{(k)}(xy)\csch^2(y)dy \nonumber \\
    I''(x) &= \frac{\sech^2 x}{x}+2\frac{\log(\cosh x)-x\tanh{x}}{x^3}\label{second_deriv}\\
    I'''(x) &= \frac{6(x\tanh x-\log(\cosh x))-3x^2\sech^2 x-2x^3\tanh x\sech^2 x}{x^4} \label{I3}
\end{align}
The Laurent series for $\cot(x)$ is of the form $\frac{1}{x}-\sum\limits_{n=0}^\infty a_nx^{2n+1}$, where the $a_n$ are all positive, with convergence over the domain of interest. Thus, $\tilde \psi'(x)$ is decreasing and concave down on $(0,1)$. Thus, if $I''(x)$ is negative, it follows that $\tilde \varphi'(x)-\frac{\log 2}{2}\tilde \psi'(x)$ is increasing. By manipulating (\ref{second_deriv}), we see that $I''(x)$ is negative if and only if 
\begin{equation*}
    \frac{x^2}{2}\le x\sinh x \cosh x-\log(\cosh x)\cosh^2 x
\end{equation*}
This inequality is saturated at the origin. Thus, if the relation holds for the derivative of each side, we may recover the initial inequality by integration. Taking derivatives, it suffices to show
\begin{equation*}
    x\le x\cosh(2x)-\sinh(2x)\log(\cosh x)
\end{equation*}

Again, this inequality is saturated at the origin. Thus, we differentiate again and would like to show
\begin{equation*}
    1\le2x\sinh(2x)-2\cosh(2x)\log(\cosh x)+1
\end{equation*}
or equivalently
\begin{equation*}
    0\le x\tanh(2x)-\log(\cosh x)
\end{equation*}
This inequality is also saturated at $x=0$. We differentiate one last time and seek to show
\begin{equation}
    0\le\tanh(2x)-\tanh(x)+2x\sech^2(2x)\label{final_increasing}
\end{equation}
But $\tanh(x)$ is increasing, so $\tanh(2x)\ge\tanh(x)$. Thus, (\ref{final_increasing}) is valid for $x\ge 0$ and we have shown that $\tilde \varphi'(x)-\frac{\log 2}{2}\tilde \psi'(x)$ is increasing on $(0,1)$. 
\par
To verify convexity we must show $\tilde \varphi'''(x)-\frac{\log 2}{2}\tilde \psi'''(x)\ge0$. The representation in equation~(\ref{varphi_rep}) demonstrates that $\tilde \varphi$ is a rescaled form of $\varphi$ with the pole at zero removed. Similarly, we now obtain a form of $\tilde \psi'$ which demonstrates that we have removed the pole of $\psi$ at zero. This process was carried out in \cite{GM} using the following representation of $\csc^2(x)$
\begin{equation*}
    \csc^2(x) = \sum\limits_{k=-\infty}^\infty\frac{1}{(k\pi+x)^2}
\end{equation*}
Thus, 
\begin{equation*}
    \tilde \psi'(x) = -\frac{2}{\pi}\sum\limits_{k\ne 0}\frac{1}{(2k+x)^2}
\end{equation*}
This series converges uniformly on $(0,\pi)$, so we may commute differentiation and summation to obtain
\begin{equation*}
    \tilde \psi'''(x) = -\frac{12}{\pi}\sum\limits_{k\ne 0}\frac{1}{(2k+x)^4}
\end{equation*}
Thus, we must show that
\begin{equation*}
    -\frac{1}{\pi}\int\limits_0^\infty y^4I^{(3)}(xy)\csch^2(y)dy\ge-\frac{6\log 2}{\pi} \sum\limits_{k\ne 0}\frac{1}{(2k+x)^4}\\
\end{equation*}
It will be sufficient to show that
\begin{equation*}
    \int\limits_0^\infty y^4I^{(3)}(xy)\csch^2(y)dy\le\frac{6\log 2}{(x-2)^4}
\end{equation*}
We will use the following lemma:

\begin{lemma}\label{xon6}
    $I'''(x)\le\frac{x}{6}$ for $x>0$.
\end{lemma}

Assuming the lemma, we can finish the proof of Lemma~\ref{log2_lemma}. 

In \cite{BM}, Boyadzhiev and Moll demonstrated that $\int\limits_0^\infty y^5\csch^2(y)dy=\frac{15\zeta(5)}{2}$. Thus,
\begin{align*}
    \int\limits_{0}^\infty y^4 I'''(xy)\csch^2(y)dy &\le \frac{1}{6}x\int\limits_0^{\infty}y^5\csch^2(y)dy\\
    &=\frac{1}{6}\cdot\frac{15\zeta(5)}{2}x\\
    &\le\frac{5\zeta(4)}{4}x\\
    &=\frac{\pi^4}{72}x\\
    &\le\frac{3}{2}x\\
\end{align*}
The line $y=\frac{3125}{2048}x$ is tangent to the graph of the convex function $\frac{4}{(x-2)^4}$. Thus, $\frac{3}{2}x\le\frac{3125}{2048}x\le\frac{4}{(x-2)^4}\le\frac{6\log 2}{(x-2)^4}$, which concludes the proof of Lemma~\ref{log2_lemma}.
\end{proof}

All that remains to be done is the proof of Lemma~\ref{xon6}.

\begin{proof}[Proof of Lemma~\ref{xon6}]
Using (\ref{I3}), we must establish
\begin{equation*}
    6(x\tanh x-\log(\cosh x))-3x^2\sech^2 x-2x^3\tanh x\sech^2 x -\frac{x^5}{6} \le 0.
\end{equation*}

This inequality is saturated at the origin. Thus, taking derivatives and manipulating the result algebraically, it suffices to show
\begin{equation}\label{coshless2}
    2(\cosh(2x)-2)-\frac{5}{6}x\cosh^4x\le 0.
\end{equation}
If $\cosh(2x)<2$, inequality~(\ref{coshless2}) is clearly valid. Let $x_0=\frac{1}{2}\log(2+\sqrt{3})$ be the unique positive value where $\cosh(2x_0)=2$. Note that $x_0\in(\frac{1}{2},1)$. If $x\ge1$, we get
\begin{align*}
     \frac{2(\cosh(2x)-2)}{x}-\frac{5}{6}\cosh^4x&\le 2(\cosh(2x)-2)-\frac{5}{6}\cosh^4x\\
     &=2(2\cosh^2(x)-3)-\frac{5}{6}\cosh^4 x
\end{align*}
But the polynomial $2(2u-3)-\frac{5}{6}u^2$ has no real roots. Thus, (\ref{coshless2}) is valid for $x\ge 1$. Now, suppose $x\in(\frac{1}{2},1)$. Then, we would like to show
\begin{equation*}
    2(2\cosh^2x-3)-\frac{5}{6}x\cosh^4 x \le 0
\end{equation*}
Now consider the polynomial $2(2u-3)-\frac{5}{6}xu^2$. If $x>\frac{4}{5}$ this has no real roots. Otherwise, it has the solutions $u=\frac{12\pm6\sqrt{4-5x}}{5x}$. Thus, if we show that
\begin{equation}\label{ineq1}
    \cosh^2(x) \le \frac{12-6\sqrt{4-5x}}{5x}
\end{equation}
for $x\in(\frac{1}{2},\frac{4}{5})$, we will complete the verification of Lemma~\ref{xon6} for the entire range of $x>0$. By computing derivatives we can see that the the right-hand-side of (\ref{ineq1}) is increasing. Thus, for $x\in(\frac{1}{2},\frac{4}{5})$,
\begin{equation*}
    \frac{12-6\sqrt{4-5(\frac{1}{2})}}{5\left(\frac{1}{2}\right)}\le\frac{12-6\sqrt{4-5x}}{5x}
\end{equation*}
Additionally, for $x\in (\frac{1}{2},\frac{4}{5}).$ 
\begin{equation*}
    \cosh^2(x)\le\cosh^2\left(\frac{4}{5}\right)
\end{equation*}
It may be verified that
\begin{equation*}
    \cosh^2\left(\frac{4}{5}\right)\le\frac{12-6\sqrt{4-5(\frac{1}{2})}}{5\left(\frac{1}{2}\right)}
\end{equation*}
from which the desired inequality follows immediately.
    
\end{proof}

\bibliographystyle{amsplain}

\begin{thebibliography}{10}

\bibitem {BG} M. van den Berg, P. Gilkey \textit{Heat content asymptotics of a Riemannian manifold with boundary},
J. Funct. Anal.  {\bf 120}, 48-71, 1994.

\bibitem {BS} M. van den Berg, S. Srisatkunarajah \textit{Heat flow and Brownian motion for a region in $\mathbb{R}^2$ with a polygonal boundary},  Probab. Theory Related Fields, {\bf 87}, 41--52, 1990.

\bibitem {BM} K. Boyadzhiev and V. Moll \textit{The integrals in Gradshteyn and Ryzhik.
Part 21: Hyperbolic functions},
SCIENTIA Series A: Mathematical Sciences, {\bf 22}, 109-127, 2011.

\bibitem {D} C. Durso \textit{On the inverse spectral problem for polygonal domains},
Ph.D. thesis, MIT 1988.

\bibitem {G} P. Gilkey \textit{Heat Content, Heat Trace, and Isospectrality}, Contemp. Math. {\bf 491}, 115--124, 2009.


\bibitem {GM} D. Grieser, S. Maronna \textit{Hearing the Shape of a Triangle}, Not. Amer. Math. Soc., {\bf 60},  1440--1447, 2013.

\bibitem {MM} P. McDonald, R. Meyers \textit{Heat content and Dirichlet spectrum},
J. Funct. Anal.  {\bf 200}, 150-159, 2003.


\end{thebibliography}

\end{document}